\documentclass[11pt,reqno]{amsart}
\usepackage{graphicx}
\usepackage{amssymb}
\usepackage{cite}
\usepackage{amsmath}
\usepackage{latexsym}
\usepackage{amscd}
\usepackage{amsthm}
\usepackage{mathrsfs}
\usepackage{url}
\usepackage[linktocpage=true,colorlinks,citecolor=blue,linkcolor=blue,urlcolor=blue]{hyperref}

\newtheorem{thm}{Theorem}[section]

\newtheorem{lem}[thm]{Lemma}
\newtheorem{prop}[thm]{Proposition}

\theoremstyle{definition}

\theoremstyle{remark}

\newtheorem{rem}{Remark}[section]

\numberwithin{equation}{section}
\def\R{\mathbb R}

\def\H{\mathbb H}
\def\SS{\mathbb S}

\def\td{\tilde}
\def\ra{\rightarrow}
\def\pt{\partial}

\begin{document}
\title[Inverse mean curvature flow]{On inverse mean curvature flow in Schwarzschild space and Kottler space}
\author{Haizhong Li}
\address{Department of Mathematical Sciences, Tsinghua University, 100084, Beijing, P. R. China}
\email{hli@math.tsinghua.edu.cn}
\author{Yong Wei}
\address{Department of Mathematics, University College London, Gower Street, London, WC1E 6BT, United Kingdom}
\email{yong.wei@ucl.ac.uk}
\thanks{The first author was supported by NSFC No. 11271214, the second author was supported by Jason D Lotay's EPSRC grant EP/K010980/1.}
\maketitle

\begin{abstract}
In this paper, we first study the behavior of inverse mean curvature flow in Schwarzschild manifold. We show that if the initial hypersurface $\Sigma$ is strictly mean convex and star-shaped, then the flow hypersurface $\Sigma_t$ converges to a large coordinate sphere as $t\ra \infty$ exponentially. We also describe an application of this convergence result. In the second part of this paper, we will analyse the inverse mean curvature flow in Kottler-Schwarzchild manifold. By deriving a lower bound for the mean curvature on the flow hypersurface independently of the initial mean curvature, we can use an approximation argument to show the global existence and regularity of the smooth inverse mean curvature flow for star-shaped and weakly mean convex initial hypersurface, which generalizes Huisken-Ilmanen's result \cite{HI08}.
\end{abstract}

\section{Introduction}\label{sec:intro}

In this paper, we analyse the inverse mean curvature flow in Schwarzschild manifold and Kottler-Schwarzchild manifold. The Schwarzschild manifold is an $n$-dimensinal ($n\geq 3$) manifold $M=[s_0,\infty)\times\SS^{n-1}$ equipped with the metric
\begin{equation}\label{schwarz}
    \bar{g}=\frac 1{1-2ms^{2-n}}ds^2+s^2g_{\SS^{n-1}},
\end{equation}
where $m>0$ is a constant, $s_0$ is the unique positive solution of $1-2ms_0^{2-n}=0$ and $g_{\SS^{n-1}}$ is the canonical round metric on the unit sphere $\SS^{n-1}$. The metric \eqref{schwarz} is asymptotically flat in the sense that the sectional curvature of $(M,\bar{g})$ approaches zero near infinity. By a change of variable, the metric \eqref{schwarz} can be written as the following warped product metric
\begin{equation}\label{Sch-model-3}
    \bar{g}=dr^2+\lambda^2(r)g_{\SS^{n-1}},
\end{equation}
where $\lambda(r): [0, \infty) \ra [s_0, \infty)$ satisfies
\begin{equation}\label{lambda'}
    \lambda'(r)=\sqrt{1-2m\lambda(r)^{2-n}}.
\end{equation}

Given a hypersurface $\Sigma$ in $(M,\bar{g})$, we say that it is strictly mean convex if its mean curvature $H$ is positive everywhere on $\Sigma$, and we say that $\Sigma$ is star-shaped if it can be represented as a radial graph over the sphere $\SS^{n-1}$, which is equivalent to that the support function $\chi=\langle \lambda(r) \pt_r,\nu\rangle>0$ everywhere on $\Sigma$. The inverse mean curvature flow of $\Sigma$  is a family of smooth embeddings $X:\Sigma\times [0,T)\ra (M^n,\bar{g})$ satisfying
\begin{equation}\label{icf}
    \pt_tX=\frac 1H\nu,
\end{equation}
where $\nu$, $H$ are the unit outward normal vector and mean curvature of the flow hypersurface $\Sigma_t=X(\Sigma,t)$ respectively.

Our first main result is following long-time existence and convergence result of the inverse mean curvature flow \eqref{icf} for strictly mean convex and star-shaped hypersurface in $(M, \bar{g})$.
\begin{thm}\label{mainthm-ICF}
The inverse mean curvature flow \eqref{icf} starting from a strictly mean convex and star-shaped hypersurface $\Sigma$ in the Schwarzschild manifold $(M, \bar{g})$ will exist for all time. The flow hypersurface $\Sigma_t$ converge to infinity while preserving strictly mean convex and star-shapedness. Moreover, there exist positive constants $\bar{\lambda}$ and $\beta'$ such that the induced metric on $\Sigma_t$ satisfies
\begin{equation}
  e^{-\frac{2t}{n-1}}g_{ij}\ra~\bar{\lambda}^2\sigma_{ij}
\end{equation}
exponentially fast and the second fundamental form $h_i^j$ satisfies
\begin{equation}
  |\frac{\lambda}{\lambda'} h_i^j-\delta_i^j|=O(e^{-\beta' t})
\end{equation}
as $t\ra+\infty$, where $\sigma_{ij}$ denotes the components of the round metric $g_{\SS^{n-1}}$. In other words, the flow $\Sigma_t$ converges to a large coordinate sphere as $t\ra\infty$.
\end{thm}

The convergence results of inverse mean curvature flow in Euclidean space $\R^n$ has been studied by Gerhardt \cite{Ge90} (see also Urbas \cite{urbas}), in the hyperbolic space $\H^n$ by Gerhardt \cite{Ge} and Ding \cite{Di}, and in the sphere $\SS^n$ by Gerhardt \cite{Ge-2013} and Makowski-Scheuer \cite{Mak-Sch}. Huisken-Ilmanen \cite{HI01} considered the weak solution of the inverse mean curvature flow in asymptotically flat manifold (in a level-set formulation), which includes the Schwarzchild manifold as a special case. Recently, Brendle-Hung-Wang \cite{BHW} investigated the inverse mean curvature flow in anti-de Sitter-Schwarzschild manifold which is asymptotically hyperbolic at the infinity, and applied the convergence result to prove a sharp Minkowski inequality for strictly mean convex and star-shaped hypersurface in anti-de Sitter-Schwarzschild manifold. Our Theorem \ref{mainthm-ICF} provides a convergence result of the inverse mean curvature flow in the Schwarzchild manifold which is asymptotically flat.

The inverse mean curvature flow is a useful tool and has many applications in geometric analysis and general relativity, see, e.g. \cite{BN,HI01,LimaG,LimaG-2,GWW2,GWWX,GL,Lee-Neves,LiWX,WX}. In \S \ref{sec:pf-ineqn}, we will describe an application of Theorem \ref{mainthm-ICF} and give a direct proof of the sharp Minkowski type inequality (due to Brendle-Hung-Wang \cite{BHW}) for strictly mean convex and star-shaped hypersurface in the Schwarzchild manifold. Moreover, by an approximation argument, we can show that such inequality also holds for weakly mean convex and star-shaped hypersurface. After the first version of this paper appeared on the arxiv, Scheuer \cite{Scheuer13} investigated the inverse mean curvature flow in warped cylinders of non-positive radial curvature, which can be viewed as a generalization of Theorem \ref{mainthm-ICF} and \cite[Theorem 1.1]{BHW}. However, he didn't give the application of the convergence result.

\vskip 2mm
In the second part of this paper, we turn to the inverse mean curvature flow in the Kottler space. The Kottler space, which is also called Kottler-Schwarzchild space, is an analogue of the Schwarzchild space in the context of asymptotically locally hyperbolic manifolds. Fix $\kappa=0, \pm 1$ and suppose that $(N^{n-1}, \hat{g})$ is a closed space form of constant sectional curvature $\kappa$. The Kottler space is an $n$-dimensional product manifold $M_{\kappa}=[s_{\kappa}, +\infty)\times N$ equipped with the warped product metric
\begin{equation}\label{kotter-metric}
  \bar{g}_{\kappa}=\frac 1{\kappa+s^2-2ms^{2-n}}ds^2+s^2\hat{g},
\end{equation}
where $s_{\kappa}$ is the largest positive root of $\kappa+s^2-2ms^{2-n}$. If $\kappa\geq 0$, then $m$ is always positive; if $\kappa=-1$, then $m$ can be negative and belongs to the interval $m\in [m_c,\infty)$, where $m_c=-(n-2)^{\frac{n-2}2}n^{-\frac n2}$. The Kottler space is asymptotically locally hyperbolic, as near the infinity the metric \eqref{kotter-metric} resembles the locally hyperbolic metric
\begin{equation}\label{local-hypr}
  b_{\kappa}=\frac 1{\kappa+s^2}ds^2+s^2\hat{g}.
\end{equation}
Note that when $\kappa=1$, $(N, \hat{g})$ is just the round sphere $\mathbb{S}^{n-1}$, the metric \eqref{local-hypr} is exactly the hyperbolic metric on $\R^+\times \SS^{n-1}$ and the metric \eqref{kotter-metric} is the anti-deSitter-Schwarzchild metric.  As in the Schwarzchild case, by a change of variables, the metric \eqref{kotter-metric} is equivalent to
\begin{equation}
  \bar{g}_{\kappa}=dr^2+\lambda_{\kappa}^2(r)\hat{g},
\end{equation}
where $\lambda_{\kappa}(r): [0, \infty)\ra [s_{\kappa}, \infty)$ satisfies
\begin{equation}
  \lambda_{\kappa}'(r)=\sqrt{\kappa+\lambda_{\kappa}(r)^2-2m\lambda_{\kappa}(r)^{2-n}}.
\end{equation}
The long-time existence and convergence result of smooth inverse mean curvature flow in the Kottler space $(M_{\kappa}, \bar{g}_{\kappa})$ with star-shaped and strictly mean convex initial hypersurface has been studied in \cite{BHW,GWWX}.

Recall that Huisken and Ilmanen \cite{HI08} derived a lower bound of the mean curvature for smooth star-shaped solution of inverse mean curvature flow in Euclidean space, independently of the initial mean curvature. Such estimate was combined with an approximation argument to prove  the global existence and regularity of the solution to the inverse mean curvature flow for star-shaped with weakly mean convex initial hypersurface. Our next result is a generalization of Huisken-Imanen's \cite{HI08} result to the inverse mean curvature flow in the Kottler space.
\begin{thm}\label{thm-H-geq0}
Let $X_0:\Sigma\ra (M_{\kappa}, \bar{g}_{\kappa})$ be a closed hypersurface of class $C^1$ with measurable, bounded and nonnegative weak mean curvature in the Kottler space with nonnegative mass $m$. Assume that $\Sigma_0=X_0(\Sigma)$ is star-shaped and the support function $\chi=\langle\lambda_{\kappa}(r)\pt_r,\nu\rangle$ satisfies
\begin{equation}\label{cond-star-sh}
  0<R_1\leq \chi\leq R_2
\end{equation}
where $R_1,R_2$ are positive constants. Then the inverse mean curvature flow \eqref{icf} in $(M_{\kappa}, \bar{g}_{\kappa})$ has a global smooth solution $X:\Sigma\times (0,+\infty)\ra (M_{\kappa}, \bar{g}_{\kappa})$ and $\Sigma_t=X(\Sigma,t)$ converges to $\Sigma_0$ uniformly in $C^0$ as $t\ra 0$.
\end{thm}

To prove Theorem \ref{thm-H-geq0}, the key ingredient is lemma \ref{lem-Hn-1} which established a lower bound of the mean curvature along inverse mean curvature flow. Precisely, we show that for a smooth solution of the inverse mean curvature flow which is initially star-shaped in the sense of \eqref{cond-star-sh}, there exists a uniform constant $0<c(n)<\infty$ such that the mean curvature has a lower bound
\begin{equation}\label{lowbd-H1}
   \min_{\Sigma_t}H\geq  c(n)R_1R_2^{-1}\min\{\frac 1{\sqrt{2}}t^{\frac 12},1\}.
\end{equation}
The proof of \eqref{lowbd-H1} is much simpler than that in \cite{HI08} for the Euclidean case. This is due to that the curvature of $(M_{\kappa}, \bar{g}_{\kappa})$ with nonnegative mass contributes enough negative terms in the evolution equations of related terms; while in $\R^n$ case, one need to apply the Sobolev inequality for hypersurfaces to explore the negative gradient terms to obtain $L^p$-estimate and then supremum estimate by Stampacchia iteration.  Note that the estimate \eqref{lowbd-H1} is independent on the initial mean curvature. Suppose $\Sigma_0$ is the initial hypersurface in Theorem \ref{thm-H-geq0} which has nonnegative weak mean curvature. Then $\Sigma_0$ can be approximated by a sequence of hypersurfaces $\Sigma_{0,\epsilon}, 0<\epsilon<\epsilon_0$ with positive mean curvature. Starting from each $\Sigma_{0,\epsilon}$, the inverse mean curvature flow has a global smooth solution in view of the estimate \eqref{lowbd-H1} and the higher regularity estimate (cf. \cite{BHW,GWWX}). Since the estimate \eqref{lowbd-H1} is independent of the initial mean curvature and is uniform in $\epsilon$ for each positive time $t$, by letting $\epsilon\ra 0$ we can obtain the desired global solution $\Sigma_t$ for all $0<t<\infty$ and it approaches $\Sigma_0$ uniformly as $t\ra 0$.

\begin{rem}
It's an interesting question whether the conclusion in Theorem \ref{thm-H-geq0} holds for inverse mean curvature flow in the Schwarzschild manifold, or in the Kottler-Schwarzchild manifold with negative mass. We will investigate this problem in the future.
\end{rem}

\vskip 2mm
The paper is organized as follows. In section \ref{section-2}, we collect some facts about the Schwarzschild manifold, Kottler-Schwarzchild manifold and star-shaped hypersurfaces;  In section \ref{sec:IMCF}, we establish the long-time existence and convergence result of the inverse mean curvature flow for star-shaped and strictly mean convex hypersurface in the Schwarzschild manifold, and describe an application. In section \ref{sec:H^n}, we first prove the lower bound \eqref{lowbd-H1} for the mean curvature along the inverse mean curvature flow in the Kottler space and then use an approximation argument to prove Theorem \ref{thm-H-geq0}.

\section{Preliminaries}\label{section-2}

In this section, we collect some facts about the Schwarzschild manifold,  Kottler-Schwarzchild manifold and star-shaped hypersurfaces.

\subsection{Schwarzchild manifold}
As in the introduction,  the Schwarzschild metric can be written as the following warped product metric on $\R^+\times\SS^{n-1}$
\begin{equation}
    \bar{g}=dr^2+\lambda^2(r)g_{\SS^{n-1}},
\end{equation}
where $\lambda(r):[0, \infty)\ra [s_0,\infty)$ satisfies
 \begin{align}\label{expre-lam}
    \lambda'(r)=\sqrt{1-2m\lambda(r)^{2-n}},\quad   \lambda''(r)=  m(n-2)\lambda(r)^{1-n}.
 \end{align}
Let $\theta=\{\theta^j\},j=1,\cdots,n-1$ be a coordinate system on $\SS^{n-1}$, $\pt_{\theta^i}$ be the corresponding coordinate vector field in $M$ and $\pt_r$ be the radial vector. By a direct calculation (see, e.g., \cite{Pe}), the curvature tensor of the Schwarzschild manifold $(M,\bar{g})$ has the following components
 \begin{align}
    \bar{R}_{ijkl}=&\frac{1-\lambda'^2}{\lambda^2}(\bar{g}_{ik}\bar{g}_{jl}-\bar{g}_{il}\bar{g}_{jk})=2m\lambda^{-n}(\bar{g}_{ik}\bar{g}_{jl}-\bar{g}_{il}\bar{g}_{jk}), \label{pre-cur1}\\
    \bar{R}_{irjr}=&-\frac{\lambda''}{\lambda}\bar{g}_{ij}=-m(n-2)\lambda^{-n}\bar{g}_{ij},\label{pre-cur2}
 \end{align}
 where $1\leq i,j,k,l\leq n-1$ and $\bar{g}_{ij}=\bar{g}(\pt_{\theta^i},\pt_{\theta^j})=\lambda^2g_{\SS^{n-1}}(\pt_{\theta^i},\pt_{\theta^j})=\lambda^2\sigma_{ij}$. Other components of the curvature tensor are equal to zero. By a further calculation, we have the Ricci curvature of $(M,\bar{g})$
 \begin{align}
    \overline{Ric}=&\left((n-2) \frac{1-\lambda'^2}{\lambda^2}-\frac{\lambda''}{\lambda}\right)\bar{g}-(n-2)\left(\frac{1-\lambda'^2}{\lambda^2}+\frac{\lambda''}{\lambda}\right)dr^2\nonumber\\
    =&m(n-2)\lambda^{-n}\bar{g}-mn(n-2)\lambda^{-n}dr^2\label{expre-Ric}
 \end{align}
 and the scalar curvature
 \begin{align}
    \bar{R}=(n-1)\left((n-2) \frac{1-\lambda'^2}{\lambda^2}-\frac{2\lambda''}{\lambda}\right)=0.
 \end{align}
We denote by $\bar{\nabla},\bar{\nabla}^2$ and $\bar{\Delta}$  the gradiant, Hessian and Laplacian operator on $(M,\bar{g})$. We can compute the Hessian of $\lambda'$:
 \begin{align}
   \bar{\nabla}^2\lambda'=&\frac{\lambda'\lambda''}{\lambda}\bar{g}+(\lambda'''-\frac{\lambda'\lambda''}{\lambda})dr^2\nonumber\\
   =&m(n-2)\lambda^{-n}\lambda'\bar{g}-mn(n-2)\lambda^{-n}\lambda'dr^2\label{hessian}
 \end{align}
 Thus we have
 \begin{equation}\label{lapalcian}
   \bar{\Delta}\lambda'=(n-1)\frac{\lambda'\lambda''}{\lambda}+\lambda'''=0.
 \end{equation}
Combining \eqref{expre-Ric},\eqref{hessian} and \eqref{lapalcian}, we conclude that $\lambda'$ satisfies the following static equation:
\begin{equation}\label{static}
    (\bar{\Delta}\lambda')\bar{g}-\bar{\nabla}^2\lambda'+\lambda'\overline{Ric}=0.
\end{equation}

\subsection{Kottler-Schwarzchild manifold}
Similarly, by a change of variables $r=r(s)$ with
\begin{equation*}
  r'(s)=\frac 1{\sqrt{\kappa+s^2-2ms^{2-n}}}, \quad r(s_{\kappa})=0,
\end{equation*}
the Kottler metric \eqref{kotter-metric} is equivalent to the following warped product metric
\begin{equation}
    \bar{g}_{\kappa}=dr^2+\lambda_{\kappa}^2(r)\hat{g},
\end{equation}
on $[0,+\infty)\times N$,  where $\lambda_{\kappa}(r):[0,\infty)\ra [s_{\kappa}, \infty)$ is the inverse of $r(s)$ and satisfies
\begin{align}
  \lambda_{\kappa}'(r)=&\sqrt{\kappa+\lambda_{\kappa}(r)^2-2m\lambda_{\kappa}(r)^{2-n}}, \label{lambda_k1}\\
  \lambda_{\kappa}''(r)=&\lambda_{\kappa}(r)+m(n-2)\lambda_{\kappa}^{1-n}(r)
\end{align}
Let $\theta=\{\theta^j\},j=1,\cdots,n-1$ be a coordinate system on $N$, $\pt_{\theta^i}$ be the corresponding coordinate vector field in $M_{\kappa}$ and $\pt_r$ be the radial vector. As before, the curvature tensor of $(M_{\kappa},\bar{g}_{\kappa})$ has the following components
 \begin{align}
    \bar{R}_{ijkl}=&\frac{1-\lambda_{\kappa}'^2}{\lambda_{\kappa}^2}(\bar{g}_{ik}\bar{g}_{jl}-\bar{g}_{il}\bar{g}_{jk})=(2m\lambda_{\kappa}^{-n}-1)(\bar{g}_{ik}\bar{g}_{jl}-\bar{g}_{il}\bar{g}_{jk})\label{cur1-kapa}\\
    \bar{R}_{irjr}=&-\frac{\lambda_{\kappa}''}{\lambda_{\kappa}}\bar{g}_{ij}=-(1+m(n-2)\lambda_{\kappa}^{-n})\bar{g}_{ij},\label{cur2-kappa}
 \end{align}
 where $1\leq i,j,k,l\leq n-1$ and $\bar{g}_{ij}=\bar{g}_{\kappa}(\pt_{\theta^i},\pt_{\theta^j})=\lambda^2\hat{g}(\pt_{\theta^i},\pt_{\theta^j})=\lambda^2\hat{g}_{ij}$. Other components of the curvature tensor are equal to zero. The Ricci curvature of $(M_{\kappa},\bar{g}_{\kappa})$ can be expressed as
 \begin{align}\label{expre-Ric-kappa}
    \overline{Ric}=&\left(m(n-2)\lambda_{\kappa}^{-n}-(n-1)\right)\bar{g}_{\kappa}-mn(n-2)\lambda_{\kappa}^{-n}dr^2
 \end{align}
 and the scalar curvature
 \begin{align}
    \bar{R}=(n-1)\left((n-2) \frac{1-\lambda_{\kappa}'^2}{\lambda_{\kappa}^2}-\frac{2\lambda_{\kappa}''}{\lambda_{\kappa}}\right)=-n(n-1).
 \end{align}
One can also check that $\lambda_{\kappa}'$ satisfies the static equation:
\begin{equation}
    (\bar{\Delta}\lambda_{\kappa}')\bar{g}_{\kappa}-\bar{\nabla}^2\lambda_{\kappa}'+\lambda_{\kappa}'\overline{Ric}=0.
\end{equation}

\subsection{Star-shaped hypersurface}\label{sec:2-3}
If $\Sigma$ is a smooth closed hypersurface in the Schwarzchild manifold $(M,\bar{g})$, we say that $\Sigma$ is star-shaped if the support function $\chi=\langle \lambda\pt_r,\nu\rangle>0$ on $\Sigma$, which implies that $\Sigma$ could be parameterized by a graph
 \begin{equation*}
    \Sigma=\{(r(\theta),\theta):\theta\in\SS^{n-1}\}
 \end{equation*}
 for a smooth function $r=r(\theta)$ on $\SS^{n-1}$. As in \cite{BHW,Di,Ge}, we define a function $\varphi$ on $\SS^{n-1}$ by $\varphi(\theta)=\Phi(r(\theta))$, where $\Phi(r)$ is a positive function satisfying $\Phi'(r)=1/{\lambda(r)}$. Define
\begin{equation*}
    v=\sqrt{1+|D\varphi|^2_{\SS^{n-1}}},
\end{equation*}
where $D$ denotes the Levi-Civita connection on $\SS^{n-1}$. Let $\theta=\{\theta^j\},j=1,\cdots,n-1$ be a coordinate system on $\SS^{n-1}$. The unit normal vector of this hypersurface could be written as
\begin{align}\label{normal-vf}
    \nu=\frac 1v(\pt_r-\frac{r^j}{\lambda^2}\pt_{\theta^j}).
\end{align}
We can express the metric and second fundamental form of $\Sigma$ as following
\begin{align}\label{expre-g}
    g_{ij}=\lambda^2(\sigma_{ij}+\varphi_i\varphi_j),\qquad h_{ij}=\frac{\lambda'}{v\lambda}g_{ij}-\frac{\lambda}{v}\varphi_{ij},
\end{align}
where $\varphi_i,\varphi_{ij}$ are covariant derivatives of $\varphi$ with respect to the metric $g_{\SS^{n-1}}$. After lifting the indice $j$, we have
\begin{equation}\label{expre-h}
    h_i^j=h_{ik}g^{kj}=\frac 1{v\lambda}\left(\lambda'\delta_i^j-\tilde{\sigma}^{jk}\varphi_{ki}\right),
\end{equation}
where $\tilde{\sigma}^{jk}=\sigma^{jk}-\frac{\varphi^j\varphi^k}{v^2}$ with $\varphi^j=\sigma^{jk}\varphi_k$. The mean curvature $H$ then has the form
\begin{equation}\label{expre-H1}
  H=\frac{(n-1)\lambda'-\tilde{\sigma}^{ij}\varphi_{ij}}{v\lambda}.
\end{equation}
The above definitions and formulas can also be similarly given for hypersurfaces in Kottler-Schwarzchild manifold $(M, \bar{g}_{\kappa})$.

\section{Inverse mean curvature flow in Schwarzchild manifold}\label{sec:IMCF}

The inverse mean curvature flow in the Schwarzschild manifold is a family $X:\Sigma\times [0,T)\ra(M,\bar{g})$ satisfying
\begin{equation}\label{IMCF}
    \pt_tX=\frac 1H\nu,
\end{equation}
where $\nu$ is the unit outward normal and $H$ is the mean curvature of $\Sigma_t=X(\Sigma,t)$. If the initial hypersurface is star-shaped and strictly mean convex, the short time existence result (see, e.g.,\cite{Ge-2006}) of \eqref{IMCF} implies the flow exists on a maximum time interval $[0,T)$. Thus it remains to study the long time behavior of the flow \eqref{IMCF}.

Let $\pt_i, i=1,2,\cdots,n-1$ be coordinate vector fields on $\Sigma_t$. Denote by $g_{ij}$ and $h_{ij}$ the components of the first and second fundamental form, by $H=g^{ij}h_{ij}$ the mean curvature and $|A|^2=h_{ik}h_{lj}g^{il}g^{jk}$ the squared norm of the second fundamental form, and by $d\mu_t$ the area element on $\Sigma_t$. We first collect the evolution equations for various geometric quantities under the inverse mean curvature flow.

\begin{lem}[Evolution equations]\label{lem-evolu}
Under the flow \eqref{IMCF}, we have
\begin{align}
  \pt_tg_{ij} =& 2H^{-1}h_{ij}, \label{evl-metric}\\
  \pt_td\mu_t=&d\mu_t,\label{evl-measure}\\
  \pt_t\nu= & \frac 1{H^2}\nabla H,\label{evl-nu}
\end{align}
\begin{align}
  \pt_th_i^j=&\frac{\Delta h_i^j}{H^2}+\frac{|A|^2}{H^2}h_i^j-\frac 2Hh_i^kh_k^j-\frac 2{H^3}\nabla_iH\nabla^jH-\frac 2H\bar{R}_{\nu i\nu k}g^{kj}\nonumber\\
  &\quad+\frac 2{H^2}g^{lj}g^{ks}h_k^m\bar{R}_{misl}+\frac 1{H^2}g^{lj}g^{ks}h_i^m\bar{R}_{mksl}+\frac 1{H^2}g^{lj}g^{ks}h_l^m\bar{R}_{mksi}\nonumber\\
  &\quad +\frac 1{H^2}\overline{Ric}(\nu,\nu)h_i^j+\frac 1{H^2}g^{lj}g^{ks}\left(\bar{\nabla}_k\bar{R}_{\nu isl}+\bar{\nabla}_l\bar{R}_{\nu ksi}\right)\label{evl-h-0}\\
  \pt_tH=&\frac{\Delta H}{H^2}-2\frac{|\nabla H|^2}{H^3}-\frac{|A|^2}H-\frac{\overline{Ric}(\nu,\nu)}H\label{evl-H}\\
  \pt_t\chi=&\frac 1{H^2}\Delta \chi+\frac{|A|^2}{H^2}\chi-\frac 1{H^2}\overline{Ric}(\nu,\pt_k)\langle\lambda\pt_r,\pt_j\rangle g^{kj}.\label{evl-chi}
\end{align}
where $\nabla$ and $\Delta$ are gradient and Laplacian operator with respect to the induced metric on the flow hypersurface $\Sigma_t$.
\end{lem}
\proof
The evolution equations for the metric, area element, unit normal and the second fundamental form can be calculated in a standard way as in\cite{Hu}, we omit the argument here. For the support function, we have
\begin{equation*}
  \pt_t\chi=\pt_t\langle\lambda\pt_r,\nu\rangle=\langle \bar{\nabla}_{\pt_t}(\lambda\pt_r),\nu\rangle+\langle \lambda\pt_r,\pt_t\nu\rangle=\frac{\lambda'}H+\frac 1{H^2}\langle\lambda\pt_r,\nabla H\rangle,
\end{equation*}
where we used the conformal property of the vector field $\lambda\pt_r$ (see \cite{B}) and \eqref{evl-nu}. On the other hand, by using the conformal property of $\lambda\pt_r$ again and the Codazzi equations, we have
\begin{equation*}
  \nabla_i\chi=\langle\bar{\nabla}_{\pt_i}(\lambda\pt_r),\nu\rangle+\langle\lambda\pt_r,\bar{\nabla}_{\pt_i}\nu\rangle=\langle \lambda\pt_r,h_i^k\pt_k\rangle
\end{equation*}
and
\begin{align*}
  \nabla_j\nabla_i\chi =&\lambda'h_{ij}-h_i^kh_{kj}\chi+\langle\lambda\pt_r,\nabla_jh_i^k\pt_k\rangle\\
  =&\lambda'h_{ij}-h_i^kh_{kj}\chi+\langle\lambda\pt_r,\nabla^kh_{ij}\pt_k\rangle++\langle\lambda\pt_r,\overline{R}_{\nu ilj}g^{lk}\pt_k\rangle.
\end{align*}
Thus we obtain
\begin{equation*}
  \Delta\chi=\lambda'H-|A|^2\chi+\langle\lambda\pt_r,\nabla H\rangle+\overline{Ric}(\nu,\pt_l)\langle\lambda\pt_r,\pt_k\rangle g^{kl}.
\end{equation*}
Then \eqref{evl-chi} follows from combining the above equations.
\endproof
\begin{rem}\label{rem-3-1}
The evolution equations \eqref{evl-metric}--\eqref{evl-H} in fact hold for inverse mean curvature flow in general Riemannian manifold; the evolution equation \eqref{evl-chi} also holds for inverse mean curvature flow in Kottler space.
\end{rem}

We could use the evolution equation \eqref{evl-chi} of the support function to show that under the inverse mean curvature flow, the evolved hypersurface $\Sigma_t$ remains star-shaped.
\begin{lem}\label{lem-star-shape}
If the initial hypersurface satisfies
\begin{equation*}
  0<R_1\leq \chi\leq R_2,
\end{equation*}
then the solution of the inverse mean curvature flow \eqref{IMCF} satisfies
\begin{equation}\label{lem-star-shape-2}
  R_1e^{\frac t{n-1}}\leq \chi\leq \lambda(r)\leq R_2e^{\frac t{n-1}}.
\end{equation}
\end{lem}
\proof
From the expression $\eqref{expre-Ric}$ for Ricci curvature of the Schwarzschild manifold, we have
 \begin{align*}
\overline{Ric}(\nu,\pt_k)\langle\lambda\pt_r,\pt_j\rangle g^{kj}=&-mn(n-2)\lambda^{-n}\langle\pt_r,\nu\rangle\langle\pt_r,\pt_k\rangle\langle\lambda\pt_r,\pt_j\rangle g^{kj}\\
=&-mn(n-2)\lambda^{-n}\chi\langle\pt_r,\pt_k\rangle\langle\pt_r,\pt_j\rangle g^{kj}\\
=& -mn(n-2)\lambda^{-n}\chi|\pt_r^{\top}|^2,
 \end{align*}
where
\begin{equation*}
  |\pt_r^{\top}|^2=\sum_{k,j=1}^{n-1}\langle\pt_r,\pt_k\rangle\langle\pt_r,\pt_j\rangle g^{kj}
\end{equation*}
is squared norm of the tangential part $\pt_r^{\top}$ of the radial vector $\pt_r$. Substituting the above into the evolution equation \eqref{evl-chi} of $\chi$, we have
\begin{equation}\label{chi-evl1}
  \pt_t\chi=\frac 1{H^2}\Delta \chi+\frac{|A|^2}{H^2}\chi+\frac 1{H^2}mn(n-2)\lambda^{-n}|\pt_r^{\top}|^2\chi,
\end{equation}
Since $\chi>0$ on the initial hypersurface $\Sigma_0$, in view of the inequality $|A|^2\geq H^2/{(n-1)}$, $m>0$ and using the parabolic maximum principle, we conclude that
\begin{equation}\label{chi-lowerbd}
  \chi\geq e^{\frac t{n-1}}\min_{\Sigma_0}\chi=e^{\frac t{n-1}}R_1
\end{equation}
which implies the star-shapedness of $\Sigma_t$.

On the other hand, since $\lambda\pt_r$ is a conformal vector field, we have
\begin{equation*}
  \pt_t\lambda^2=2\pt_t\langle \lambda\pt_r, \lambda\pt_r\rangle=2\langle \lambda\pt_r, \frac {\lambda'}H\nu\rangle=2\frac {\lambda'}H\chi.
\end{equation*}
Note that at the point of a hypersurface in $(M, \bar{g})$ most distant from the origin, we have $\chi=\lambda$ and $H\geq (n-1)\lambda'/\lambda$ (see \cite[Lemma 2.1]{LiWX-2}). Then
\begin{equation*}
  \frac d{dt}\max_{\Sigma_t}\lambda(r)^2\leq \frac 2{n-1}\max_{\Sigma_t}\lambda(r)^2,
\end{equation*}
which implies that
\begin{equation}\label{chi-uperbd}
  \chi\leq \lambda(r)\leq R_2e^{\frac t{n-1}}.
\end{equation}
\endproof
\begin{rem}\label{rem-3-2}
The conclusion in lemma \ref{lem-star-shape} is also true for inverse mean curvature flow in Kottler space. In fact, at the point $x\in \Sigma_t$ where $\chi(x,t)=\min_{\Sigma_t}\chi$, we have $\pt_r\perp T_x\Sigma_t$ and therefore $|\pt_r^{\top}|=0$. Then we can still derive the lower bound \eqref{chi-lowerbd} from \eqref{chi-evl1}. The proof of the upper bound for $\chi$ is similar.
\end{rem}

The flow equation \eqref{IMCF} is often called the parametric form of the flow. Since each $\Sigma_t$ is star-shaped, it can be represented as a graph
\begin{equation*}
    \Sigma_t=\{(r(\theta,t),\theta):\theta\in\SS^{n-1}\}
\end{equation*}
Then the flow equation is equivalent to the following non-parametric form of the flow (cf.\cite{BHW,Di,Ge})
\begin{equation}\label{icf-non-pa}
    \frac{\pt r}{\pt t}=\frac vH.
\end{equation}
The speed function $v/H$ depends on $r,Dr,D^2r$. It is easy to see that the flow equation \eqref{icf-non-pa} is parabolic. In deed, as in \S \ref{sec:2-3} we define $\varphi(\theta,t)=\Phi(r(\theta,t))$ with $\Phi(r)$ is a positive function satisfying $\Phi'(r)=1/{\lambda(r)}$. Then
\begin{equation*}
  \varphi_i=\frac{r_i}{\lambda},\quad \varphi_{ij}=\frac{r_{ij}}{\lambda}-\frac{\lambda'r_ir_j}{\lambda^2},
  \end{equation*}
and we have
\begin{equation}\label{expre-H2}
  H=\frac{(n-1)\lambda'}{v\lambda}-\frac{\tilde{\sigma}^{ij}}{v\lambda^2}(r_{ij}-\frac{\lambda'r_ir_j}{\lambda}).
\end{equation}
Thus
\begin{equation*}
   \frac{\pt}{\pt r_{ij}}(\frac vH)=\frac{\tilde{\sigma}^{ij}}{H^2\lambda^2}
 \end{equation*}
which is nonnegative definite and therefore \eqref{icf-non-pa} is parabolic.

\begin{lem}\label{lem-lambda}
Let $r_1,r_2$ be constants such that
\begin{equation*}
    r_1~<r(\theta)<r_2
\end{equation*}
holds on the initial hypersurface $\Sigma_0$. Then on $\Sigma_t$ we have the estimate
\begin{equation}\label{chi-bd}
    \lambda(r_1)e^{\frac t{n-1}}<\lambda(r(\theta,t))<\lambda(r_2)e^{\frac t{n-1}},\qquad \forall~t\in[0,T).
\end{equation}
\end{lem}
\proof
Let $S_{r_i}, i=1,2$ be coordinate spheres $\{r_i\}\times\SS^{n-1}$. We solve the inverse mean curvature flows with initial hypersurface $S_{r_i}$ respectively. If the initial hypersurface is a coordinate sphere, the inverse mean curvature flow becomes a scalar flow:
\begin{equation*}
    \frac{dr}{dt}=\frac 1H=\frac{\lambda}{(n-1)\lambda'},
\end{equation*}
where we used that the principal curvatures of a coordinate sphere are ${\lambda'}/{\lambda}$. Then
\begin{equation*}
    \frac{d\lambda}{dt}=\frac{\lambda}{n-1}.
\end{equation*}
From this we deduce that $\lambda(r_i(t))=\lambda(r_i(0))e^{\frac t{n-1}}$. By the parabolic maximum principle, we have
\begin{equation*}
    r_1(t)<r(\theta,t)<r_2(t),\qquad t\in [0,T).
\end{equation*}
Since $\lambda'>0$, we also have
\begin{equation*}
    \lambda(r_1(t))<\lambda(r(\theta,t))<\lambda(r_2(t)).
\end{equation*}
The assertion follows from the above inequality.
\endproof

\subsection{Proof of Theorem \ref{mainthm-ICF}}
We now give the proof of Theorem \ref{mainthm-ICF}. In this subsection, we will assume that $(M, \bar{g})$ is the standard Schwarzchild manifold. We first estimate the upper bound for the mean curvature of $\Sigma_t$.
\begin{lem}\label{lem-H-u}
There is a constant $C_1>0$ such that $He^{\frac t{n-1}}\leq C_1$, where $C_1$ depends on $m,n,r_1$ and $\max_{\Sigma_0}H$.
\end{lem}
\proof
From the evolution equation \eqref{evl-H} of $H$, we have
\begin{align}\label{H^2}
    \pt_tH^2=&-2H\Delta\frac 1H-2|A|^2-2\overline{Ric}(\nu,\nu).
\end{align}
Using \eqref{expre-Ric}, we know that $\overline{Ric}(\nu,\nu)=O(\lambda^{-n})$, which implies that $|\overline{Ric}(\nu,\nu)|\leq C\lambda(r_1)^{-n}e^{-\frac{nt}{n-1}}$ by Lemma \ref{lem-lambda}, where $C$ depends only on $m,n$. Then \eqref{H^2}, together with the inequality $|A|^2\geq H^2/(n-1)$ yields that
\begin{align*}
    \frac d{dt}H^2_{\max}\leq&-\frac{2}{n-1}H^2_{\max}+C\lambda(r_1)^{-n}e^{-\frac{nt}{n-1}}.
\end{align*}
From this, we can obtain that
\begin{equation*}
  e^{\frac {2t}{n-1}}H^2(x,t)\leq \max_{\Sigma}H^2(x,0)+\frac{n-1}{n-2}C\lambda(r_1)^{-n},
\end{equation*}
which gives the assertion of the lemma.
\endproof

To estimate the lower bound of $H$, by the definition of $\varphi$ we have
\begin{equation}\label{evl-varphi}
    \frac{\pt\varphi}{\pt t}=\frac v{\lambda H}=\frac 1{H\chi}=\frac 1F,
\end{equation}
where the function
\begin{align}\label{H-chi}
    F=H\chi=\frac{(n-1)\lambda'-\td{\sigma}^{ij}\varphi_{ij}}{v^2}.
\end{align}
The flow equation \eqref{evl-varphi} is also parabolic.

\begin{lem}\label{lem-H-l}
There is a constant $C_2>0$ such that $He^{\frac t{n-1}}\geq C_2$, where $C_2$ depends on $r_2$ and $\min_{\Sigma_0}(\lambda H/v)$.
\end{lem}
\proof
If we differentiate \eqref{evl-varphi} with respect to $t$, we obtain
\begin{align*}
    \pt_t(\pt_t\varphi)=&\frac{\tilde{\sigma}^{ij}}{v^2F^2}(\pt_t\varphi)_{ij}-\frac 1{F^2}\frac{\pt F}{\pt \varphi_i}(\pt_t\varphi)_i-\frac{(n-1)\lambda\lambda''}{\lambda^2H^2}\pt_t\varphi.
\end{align*}
Then the maximal principle implies that
\begin{align*}
    \frac d{dt}(\pt_t\varphi)_{\max}\leq&-\frac{(n-1)\lambda''}{\lambda H^2}(\pt_t\varphi)_{\max}.
\end{align*}
So from \eqref{expre-lam} we have
\begin{align}\label{var-lowerbd}
    \frac d{dt}(\pt_t\varphi)_{\max}\leq&0.
\end{align}
Noting that $\pt_t\varphi=v/(\lambda H)$ and $v\geq 1$,  we obtain that
\begin{align*}
    \lambda H\geq \min_{\Sigma_0}\frac{\lambda H}{v}>0,
\end{align*}
The assertion follows from the above inequality and Lemma \ref{lem-lambda}.
\endproof

For the first spacial derivatives of $\varphi$, we have the following estimate.
\begin{lem}\label{lem-varphi}
There is a constant $\beta>0$ such that $|D\varphi|_{\SS^{n-1}}\leq O(e^{-\beta t})$.
\end{lem}
\proof
Let $\omega=\frac 12|D\varphi|^2_{\SS^{n-1}}$. We can compute as lemma 8 in \cite{BHW} to obtain the evolution of $\omega$:
\begin{align*}
    \pt_t\omega=&\frac{\tilde{\sigma}^{ij}}{v^2F^2}\omega_{ij}-\frac 1{F^2}\frac{\pt F}{\pt \varphi_i}\omega_i-\frac{\tilde{\sigma}^{ij}}{v^2F^2}\sigma^{kl}\varphi_{ik}\varphi_{jl}\\
    &\quad -\frac{2(n-2)}{\lambda^2H^2}\omega-\frac{2(n-1)\lambda''}{\lambda H^2}\omega.
\end{align*}
By Lemma \ref{lem-lambda} and Lemma \ref{lem-H-u}, there is a constant $\beta>0$ such that
\begin{equation*}
    \frac{(n-2)}{\lambda^2H^2}\geq\beta.
\end{equation*}
So from \eqref{expre-lam} we have
\begin{align*}
    \frac d{dt}\omega_{\max}\leq-2\beta~\omega_{\max},
\end{align*}
which implies the Lemma.
\endproof

Next, we will estimate the second fundamental form of $\Sigma_t$. We define the tensor $\eta$ by (as in \cite{HI08})
\begin{equation*}
  \eta_i^j=Hh_i^j.
\end{equation*}
Combining the evolution equations \eqref{evl-h-0},\eqref{evl-H}, and using the expression \eqref{pre-cur1}--\eqref{pre-cur2} for the curvature tensor of the Schwarzchild metric and lemma \ref{lem-lambda}, we have
\begin{align}
  \pt_t\eta_i^j =&\frac{1}{H^2}\Delta \eta_i^j-\frac 2{H^3}\nabla^kH\nabla_k\eta_i^j-\frac 2{H^2}\nabla^jH\nabla_iH \nonumber\\
   & -\frac 2{H^2}\eta_i^k\eta_k^j+\frac{|\eta|}{H^2}O(e^{-\frac{nt}{n-1}})+o(e^{-\frac{n+1}{n-1}t}).
\end{align}
If on the time interval considered we have an uniform bound $H_0\leq H\leq H_1$ for the mean curvature, by Hamilton's maximum principle for parabolic system \cite{Ha}, we can bound the largest eigenvalue $\mu_{n-1}$ of $\eta$ above by the solution of the  following ODE
\begin{equation*}
  \frac d{dt}\varphi=-\frac 2{H_1^2}\varphi^2+\frac{\varphi}{H_0^2}O(e^{-\frac{nt}{n-1}})+o(e^{-\frac{n+1}{n-1}t}).
\end{equation*}
So that $\mu_{n-1}$ and then the largest principal curvature $\kappa_{n-1}$ of $\Sigma_t$ have a uniform upper bound depending on $H_1$ and $H_0$. Since the mean curvature is bounded from below, it follows that the full second fundamental form is bounded by
\begin{equation}\label{bound-A}
  |A|\leq C(n,H_0,H_1).
\end{equation}
In view of the mean curvature estimate in Lemma \ref{lem-H-u} and Lemma \ref{lem-H-l}, we have the long time existence of the inverse mean curvature flow.
\begin{prop}
The solution of the inverse mean curvature flow is defined on $[0,\infty)$.
\end{prop}
\proof
Let $[0,T)$ be the maximum time interval of existence for the smooth solution of the inverse mean curvature flow. If $T<\infty$, then Lemma \ref{lem-H-u} and Lemma \ref{lem-H-l} imply that
\begin{equation*}
 C_2e^{-\frac T{n-1}} \leq H\leq C_1.
\end{equation*}
From \eqref{bound-A}, we know that the full second fundamental form of $\Sigma_t$ is uniformly bounded as $t\ra T$. Then the regularity results of Krylov \cite{Kry} and the short time existence theorem imply that we can extend the solution smoothly beyond $T$,  contradicting with the maximum of $T$. So we conclude that $T$ must be $\infty$.
\endproof

We now finish the proof of Theorem \ref{mainthm-ICF}.

\begin{proof}[Proof of Theorem \ref{mainthm-ICF}]
It remains to show that the solution $\Sigma_t$ converges to a large coordinate sphere as $t\ra\infty$. Let us define
\begin{equation}
  \tilde{\lambda}=\lambda e^{-\frac t{n-1}}.
\end{equation}
Then $\tilde{\lambda}$ satisfies
\begin{equation}\label{evl-td-lambda}
  \pt_t\tilde{\lambda}=\frac{\lambda'v}{H}e^{-\frac t{n-1}}-\frac 1{n-1}\tilde{\lambda}\quad (:=\tilde{F}),
\end{equation}
where we denote the right hand side of \eqref{evl-td-lambda} by $\tilde{F}$. By lemma \ref{lem-lambda}, the family $\tilde{\lambda}(\cdot,t)$ is uniformly bounded. By lemma \ref{lem-H-u} and lemma \ref{lem-H-l}, $|\pt_t\tilde{\lambda}|$ is also uniformly bounded. Noting that
\begin{equation}\label{lambda-varphi}
  \tilde{\lambda}_{ij}=(\lambda''\lambda^2+\lambda\lambda'^2)e^{-\frac t{n-1}}\varphi_i\varphi_j+\lambda\lambda'e^{-\frac t{n-1}}\varphi_{ij}
\end{equation}
and the expression \eqref{expre-H1} of $H$, we deduce that
\begin{equation*}
  \frac{\pt\tilde{F}}{\pt\tilde{\lambda}_{ij}}=\frac{\tilde{\sigma}^{ij}}{H^2\lambda^2}
\end{equation*}
which is positive definite and therefore \eqref{evl-td-lambda} is parabolic. Moreover, from lemma \ref{lem-H-u}, lemma \ref{lem-H-l} and lemma \ref{lem-varphi}, we conclude that \eqref{evl-td-lambda} is uniformly parabolic.

By lemma \ref{lem-varphi}, $D\tilde{\lambda}$ decays exponentially fast:
\begin{equation}
  D\tilde{\lambda}=D\lambda e^{-\frac t{n-1}}=\lambda'\lambda e^{-\frac t{n-1}}D\varphi=O(e^{-\beta t}).
\end{equation}
Thus $\tilde{\lambda}$ converges to a positive constant $\bar{\lambda}$ uniformly. From the regularity estimate of Krylov \cite[\S 5.5]{Kry}, the second derivatives of $\tilde{\lambda}$ are uniformly bounded in $C^{0,\alpha}$. Using the interpolation theorem we deduce that $D^2\tilde{\lambda}$ also decays exponentially fast.
In view of \eqref{lambda-varphi} and lemma \ref{lem-lambda}, lemma \ref{lem-varphi}, we have $|D^2\varphi|=O(e^{-\tilde{\beta} t})$ for some constant $\tilde{\beta}>0$.

By \eqref{expre-g} and lemma \ref{lem-varphi}, the metric of $\Sigma_t$ satisfies
\begin{equation}
  e^{-\frac{2t}{n-1}}g_{ij}\ra~\bar{\lambda}^2\sigma_{ij}
\end{equation}
exponentially fast.  From the expression \eqref{expre-h} of $h_i^j$, we have
\begin{equation}
  |\frac{\lambda}{\lambda'} h_i^j-\delta_i^j|=(\frac {1}{v}-1)\delta_i^j-\frac{1}{v\lambda'}\tilde{\sigma}^{jk}\varphi_{ki}=O(e^{-\beta' t})
\end{equation}
for a positive constant $\beta'=\min\{2\beta,\tilde{\beta}\}$. This implies that $\Sigma_t$ converges to a large coordinate sphere exponentially as $t\ra\infty$ and finishes the proof of Theorem \ref{mainthm-ICF}.
\end{proof}
\subsection{Application}\label{sec:pf-ineqn}
In this subsection, we apply Theorem \ref{mainthm-ICF} to reprove the Minkowski type inequality, which was obtained by Brendle-Hung-Wang \cite{BHW} as the limit case of their inequality for strictly mean convex and star-shaped hypersurface in anti-de Sitter-Schwarzschild manifold. Although the argument is similar, the proof we give here is also interesting as the long-time behaviors of the inverse mean curvature flow in Schwarzschild and anti-de Sitter-Schwarzschild manifold are different, which makes the proof relatively simpler.

\begin{thm}[\cite{BHW}]\label{main-thm}
Let $\Sigma$ be a strictly mean convex and star-shaped closed hypersurface in the Schwarzchild manifold $(M,\bar{g})$. Then
\begin{equation}\label{main-inequ}
    \int_{\Sigma}fH d\mu\geq (n-1)\omega_{n-1}\left((\frac{|\Sigma|}{\omega_{n-1}})^{\frac{n-2}{n-1}}-2m\right),
\end{equation}
where $f(r)=\lambda'(r)$, $\omega_{n-1}$ is the area of the unit sphere $\SS^{n-1}\subset\R^n$ and $|\Sigma|$ is the area of $\Sigma$. Moreover, equality holds in \eqref{main-inequ} if and only if $\Sigma$ is a coordinate sphere $\{s\}\times\SS^{n-1}$.
\end{thm}
Recall that the boundary $\pt M=\{s_0\}\times\SS^{n-1}$ is called the horizon of the Schwarzchild manifold $(M,\bar{g})$, its area is equal to $|\pt M|=s_0^{n-1}\omega_{n-1}$. Since $s_0$ is the unique positive solution of $1-2ms_0^{2-n}=0$, we have $2m=({|\pt M|}/{\omega_{n-1}})^{\frac{n-2}{n-1}}$. Therefore \eqref{main-inequ} is equivalent to the following inequality
\begin{equation}\label{main-inequ2}
    \int_{\Sigma}fH d\mu\geq (n-1)\omega_{n-1}^{\frac 1{n-1}}\left(|\Sigma|^{\frac{n-2}{n-1}}-|\pt M|^{\frac{n-2}{n-1}}\right).
\end{equation}
The classical Minkowski inequality for convex hypersurface $\Sigma$ in $\R^n$ states that
\begin{equation}\label{mink-R}
  \int_{\Sigma}H d\mu\geq (n-1)\omega_{n-1}^{\frac 1{n-1}}|\Sigma|^{\frac{n-2}{n-1}}.
\end{equation}
This was generalized by Guan and Li \cite{GL} to weakly mean convex and star-shaped hypersurface using the inverse mean curvature flow. By letting $m\ra 0$, the Schwarchild metric reduces to the Euclidean metric $\bar{g}=ds^2+s^2g_{\SS^{n-1}}$ and the potential $f$ becomes $f=1$. Thus as a limit case Theorem \ref{main-thm} recovers the Minkowski inequality \eqref{mink-R} for strictly mean convex and star-shaped hypersurface $\Sigma$ in $\R^n$. Note that Huisken recently shown that the inequality \eqref{mink-R} also holds for outward-minimizing hypersurfaces in $\R^n$.

The proof of Theorem \ref{main-thm} follows a similar argument in \cite{GL,BHW}. We evolve the hypersurface $\Sigma$ by inverse mean curvature flow and define the quantity
\begin{equation*}
    Q(t)=|\Sigma_t|^{-\frac{n-2}{n-1}}\left(\int_{\Sigma_t}fHd\mu_t+2(n-1)m\omega_{n-1}\right)
\end{equation*}
on the flow hypersurfaces $\Sigma_t$, where $|\Sigma_t|$ is the area of $\Sigma_t$. We first show that $Q(t)$ is monotone non-increasing under the inverse mean curvature flow.
\begin{prop}\label{prop-mono}
Under the inverse mean curvature flow \eqref{IMCF}, the quantity $Q(t)$ is monotone non-increasing in $t$.
\end{prop}
\proof
As the proof of Proposition 5.3 in \cite{BHW}, we first have
\begin{align}\label{eq4-1}
    \frac d{dt}\int_{\Sigma_t}fHd\mu_t\leq&\int_{\Sigma_t}\left(\frac{n-2}{n-1}fH+2\langle\bar{\nabla}f,\nu\rangle\right)d\mu_t,
\end{align}
and equality holds if and only if $\Sigma_t$ is totally umbilical. For convenience of reader, we include the proof of \eqref{eq4-1} here.
\begin{align}
  \frac d{dt}\int_{\Sigma_t}fHd\mu_t= & \int_{\Sigma_t}\left(\pt_tfH+f\pt_tH+fH\right)d\mu_t\nonumber \\
  = & \int_{\Sigma_t}\biggl(\langle\bar{\nabla}f,\nu\rangle-f\Delta\frac 1H-\frac fH|A|^2-\frac fH\overline{Ric}(\nu,\nu)+fH\biggr)d\mu_t\nonumber\\
  \leq & \int_{\Sigma_t}\biggl(\langle\bar{\nabla}f,\nu\rangle-\frac 1H(\Delta f+f\overline{Ric}(\nu,\nu))+\frac{n-2}{n-1}fH\biggr)d\mu_t,\label{eq4-2}
\end{align}
where we used $|A|^2\geq H^2/{(n-1)}$ in the last inequality. Using the identity $\Delta f=\bar{\Delta}f-\bar{\nabla}^2f(\nu,\nu)-H\langle\bar{\nabla}f,\nu\rangle$ and \eqref{lapalcian},\eqref{static}, we have
\begin{equation*}
  \Delta f+f\overline{Ric}(\nu,\nu)=-H\langle\bar{\nabla}f,\nu\rangle.
\end{equation*}
Substituting this into \eqref{eq4-2} gives \eqref{eq4-1}. If equality holds in \eqref{eq4-1}, then $|A|^2=H^2/{(n-1)}$ and $\Sigma_t$ is totally umbilical.

Let $\Omega_t$ denote the region bounded by $\Sigma_t$ and the horizon $\pt M$. Using the divergence theorem and noting that $\bar{\Delta}f=0$, we get
\begin{align*}
    \int_{\Sigma_t}\langle\bar{\nabla}f,\nu\rangle d\mu_t=&\int_{\Omega_t}\bar{\Delta}f dvol+m(n-2)\omega_{n-1}=m(n-2)\omega_{n-1}
\end{align*}
which is a constant. Thus we obtain
\begin{align*}
    \frac d{dt}\left(\int_{\Sigma_t}fHd\mu_t+2(n-1)m\omega_{n-1}\right)\leq&\frac{n-2}{n-1}\left(\int_{\Sigma_t}fHd\mu_t+2(n-1)m\omega_{n-1}\right).
\end{align*}

On the other hand, from the evolution \eqref{evl-measure} of the area element $d\mu_t$, the area of $|\Sigma_t|$ satisfies $\frac d{dt}|\Sigma_t|=|\Sigma_t|$. So we conclude that
\begin{align*}
    \frac d{dt}Q(t)\leq&0,
\end{align*}
equality holds if and only if \eqref{eq4-1} assumes equality and then $\Sigma_t$ is totally umbilical.
\endproof

We next investigate the limit of $Q(t)$ as $t\ra\infty$.
\begin{prop}
We have
\begin{align*}
    \liminf_{t\ra\infty}Q(t)\geq&(n-1)\omega_{n-1}^{\frac 1{n-1}}.
\end{align*}
\end{prop}
\proof
From the expression \eqref{expre-g} of the metric $g$ and lemma \ref{lem-varphi}, the area element $d\mu_t$ satisfies
\begin{equation*}
    d\mu_t=\lambda^{n-1}(1+O(e^{-2\beta t}))dvol_{\SS^{n-1}},
\end{equation*}
Then we have
\begin{align}
    |\Sigma_t|^{\frac{n-2}{n-1}}=&\left(\int_{\SS^{n-1}}\lambda^{n-1}dvol_{\SS^{n-1}}\right)^{\frac{n-2}{n-1}}\left(1+O(e^{-2\beta t})\right).
\end{align}

On the other hand,
\begin{align*}
    f=\sqrt{1-2m\lambda^{2-n}}=1-m\lambda^{2-n}+O(\lambda^{4-2n}).
\end{align*}
By the expression \eqref{expre-H1} of the mean curvature $H$ and the exponentially decay of $\varphi_i,\varphi_{ij}$,
\begin{align*}
    \lambda H=&\frac{(n-1)\lambda'}{v}-\frac{\sigma^{ij}\varphi_{ij}}{v}+\frac{\varphi^i\varphi^j\varphi_{ij}}{v^3} =n-1+O(e^{-\gamma t})
\end{align*}
for some positive constant $\gamma=\min\{2\beta,\tilde{\beta},\frac{n-2}{n-1}\}$. So we have
\begin{align}
    \int_{\Sigma_t}fHd\mu_t=&(n-1)\int_{\SS^{n-1}}\lambda^{n-2}dvol_{\SS^{n-1}}\left(1+O(e^{-\gamma t})\right).
\end{align}
Then we obtain
\begin{align}
    \liminf_{t\ra\infty}Q(t)\geq&(n-1)\liminf_{t\ra\infty}\frac{\int_{\SS^{n-1}}\lambda^{n-2} dvol_{\SS^{n-1}}}{\left(\int_{\SS^{n-1}}\lambda^{n-1}dvol_{\SS^{n-1}}\right)^{\frac{n-2}{n-1}}}
\end{align}

Since $\tilde{\lambda}=\lambda e^{-\frac t{n-1}}$ converges to a constant $\bar{\lambda}$, there exists a positive function $\epsilon(t)$ such that $\lim_{t\ra\infty}\epsilon(t)=0$ and
\begin{equation*}
 \bar{\lambda}-\epsilon<\tilde{\lambda}<\bar{\lambda}+\epsilon
\end{equation*}
or equivalently
\begin{equation*}
  (\bar{\lambda}-\epsilon)e^{\frac t{n-1}}<\lambda<(\bar{\lambda}+\epsilon)e^{\frac t{n-1}}
\end{equation*}
when $t$ is sufficiently large. So we conclude that
\begin{equation}
  \liminf_{t\ra\infty}Q(t)\geq(n-1)\omega_{n-1}^{\frac 1{n-1}}\liminf_{t\ra\infty}(\frac{\bar{\lambda}-\epsilon}{\bar{\lambda}+\epsilon})^{n-2}=(n-1)\omega_{n-1}^{\frac 1{n-1}},
\end{equation}
which completes the proof.
\endproof
\begin{rem}
Note that in \cite{BHW}, Beckner's sharp Sobolev inequality on sphere plays a crucial role in estimating the monotone quantity $Q(t)$, while here we can avoid this due to the good long time behavior of the inverse mean curvature flow in Schwarzchild manifold.
\end{rem}

Now we can finish the proof of Theorem \ref{main-thm}. Since $Q(t)$ is monotone non-increasing in time $t$, we have
\begin{equation*}
  Q(0)\geq \liminf_{t\ra\infty}Q(t)\geq(n-1)\omega_{n-1}^{\frac 1{n-1}}.
\end{equation*}
Thus we obtain
\begin{equation*}
  \int_{\Sigma_t}fHd\mu_t+2(n-1)m\omega_{n-1}\geq  (n-1)\omega_{n-1}^{\frac 1{n-1}}|\Sigma_t|^{\frac{n-2}{n-1}}
\end{equation*}
which is equivalent to \eqref{main-inequ}. If the equality holds in \eqref{main-inequ}, then $Q(t)$ is a constant in $t$. From the proof of Proposition \ref{prop-mono}, the hypersurface $\Sigma_0$ is totally umbilical. It follows from the Codazzi equations that $\overline{Ric}(\nu,e_i)=0$ for any tangent vector fields $e_i$. Since $m>0$, the expression \eqref{expre-Ric} of Ricci curvature implies that the radial vector $\pt_r$ is either parallel or orthogonal to the unit normal vector $\nu$ of $\Sigma_0$. By the star-shapedness of $\Sigma_0$, $\pt_r$ is parallel to $\nu$ at each point of $\Sigma_0$ and then $\Sigma_0$ is a coordinate sphere $\{s\}\times \SS^{n-1}$. This completes the proof of Theorem \ref{main-thm}.

\vskip 2mm
Finally we show that the inequality \eqref{main-inequ} is also true for star-shaped and weakly mean convex (i.e., $H\geq 0$) hypersurface $\Sigma$, as we can approximate such $\Sigma$ by a family of star-shaped and strictly mean convex hypersurfaces $\Sigma_{\epsilon}, 0<\epsilon<\epsilon_0$ and then the conclusion follows from the approximation process. To this end, we need the following approximation lemma
\begin{lem}\label{claim-0}
Suppose that $\Sigma_0$ is a star-shaped and weakly mean-convex hypersurface in Schwarzchild manifold $(M, \bar{g})$, then we can approximate it by a family of star-shaped and strictly mean-convex hypersurfaces $\Sigma_{\epsilon}, 0<\epsilon<\epsilon_0$.
\end{lem}
\begin{proof}
Since $\Sigma_0$ is weakly mean-convex, we have that the mean curvature $H\geq 0$ on $\Sigma_0$. We solve the mean curvature flow $X: \Sigma\times [0, \epsilon_0)\ra (M, \bar{g})$,
\begin{equation}\label{MCF-Sch}
  \pt_tX=-H\nu,
\end{equation}
with $\Sigma_0=X(\Sigma,0)$ as the initial data. Then the mean curvature $H$ evolves by (see \cite{Hu-86})
\begin{equation}\label{evl-H-MCF-Sch}
  \pt_tH=\Delta H+H(|A|^2+\overline{Ric}(\nu,\nu)).
\end{equation}
Then in view of the maximum principle, \eqref{evl-H-MCF-Sch} implies that $H\geq 0$ for all $\Sigma_t, t\in [0, \epsilon_0)$. Suppose $H(x_0,t_0)=0$ at some interior point $x_0\in \Sigma_{t_0}, 0<t_0<\epsilon_0$, the strong maximum principle implies that $H(x,t)\equiv 0$ for $(x,t)\in \Sigma_t, t\in [0,t_0)$. However, since $\Sigma_t$ is a closed hypersurface in $(M, \bar{g})$,  there exists at least one elliptic point on each $\Sigma_t$ (see e.g.,\cite[Lemma 2.1]{LiWX-2}). Thus there is at least one point $x\in \Sigma_t$ such that $H(x,t)>0$, which is a contradiction. So we conclude that $H>0$ for all $\Sigma_t, t\in (0,\epsilon_0)$. Since star-shaped is an open condition, we have that $\Sigma_{\epsilon}$ is also star-shaped for $0<\epsilon<\epsilon_0$ provided that $\epsilon_0$ is sufficiently small.
\end{proof}

Suppose $\Sigma$ is a star-shaped and weakly mean convex hypersurface in $(M, \bar{g}_0)$, then lemma \ref{claim-0} gives a family of star-shaped and strictly mean-convex hypersurfaces $\Sigma_{\epsilon}, 0<\epsilon<\epsilon_0$ with $\Sigma_{\epsilon}\ra \Sigma$ smoothly as $\epsilon\ra 0$. By Theorem \ref{main-thm}, the inequality \eqref{main-inequ} holds for each $\Sigma_{\epsilon}$. Therefore by letting $\epsilon\ra 0$, we conclude that \eqref{main-inequ} also holds for $\Sigma$.
\begin{rem}
A similar argument implies that the main inequality in \cite[Theorem 1]{BHW} also holds for star-shaped and weakly mean convex hypersurface in anti-deSitter-Schwarzchild manifold.
\end{rem}

\section{Inverse mean curvature flow in Kottler space}\label{sec:H^n}

In this section, we consider the inverse mean curvature flow in the Kotter-Schwarzchild manifold $(M_{\kappa}, \bar{g}_{\kappa})$. The existence and convergence result of inverse mean curvature flow in $(M_{\kappa}, \bar{g}_{\kappa})$ with star-shaped and strictly mean convex initial hypersurface has been studied in \cite{BHW, GWWX} (see also \cite{Ge,Di} for $\H^n$ case). We will show that star-shaped and weakly mean convex initial hypersurface has a global smooth solution of strictly mean convex to the inverse mean curvature flow for all time $t>0$. Firstly, we prove the following lower bound for the mean curvature, independently of the initial mean curvature.

\begin{lem}\label{lem-Hn-1}
Suppose that $X: \Sigma\times [0,T)\ra (M_{\kappa}, \bar{g}_{\kappa})$ is a smooth star-shaped solution to the inverse mean curvature flow \eqref{icf} in the Kottler space with nonnegative mass $m\geq 0$ and the support function $\chi=\langle \lambda_{\kappa}\pt_r, \nu\rangle$ of the initial hypersurface $\Sigma_0=X_0(\Sigma)$ satisfies
\begin{equation}\label{lem-Hn-cond1}
  0<R_1\leq \chi\leq R_2.
\end{equation}
Then we have the following estimate on the mean curvature $H$ of $\Sigma_t, t\in [0,T)$.
\begin{equation}\label{H-sharp-bd}
 \min_{\Sigma_t}H\geq   e^{-\frac 1{n-1}}(n-1)^{\frac 12}R_1R_2^{-1}\min\{\frac 1{\sqrt{2}}t^{\frac 12},1\}.
\end{equation}
\end{lem}
\proof
From the expression \eqref{expre-Ric-kappa} for Ricci tensor of $(M_{\kappa}, \bar{g}_{\kappa})$, we have
\begin{align*}
 & \overline{Ric}(\nu,\nu)=m(n-2)\lambda_{\kappa}^{-n}-(n-1)-mn(n-2)\lambda_{\kappa}^{-n}|\pt_r^{\perp}|^2\\
&\overline{Ric}(\nu,\pt_k)\langle\lambda_{\kappa}\pt_r,\pt_j\rangle g^{kj}=-mn(n-2)\lambda_{\kappa}^{-n}\chi|\pt_r^{\top}|^2,
\end{align*}
Then by lemma \ref{lem-evolu} and remark \ref{rem-3-1}, we have
\begin{align}
\pt_tH=&\frac{\Delta H}{H^2}-2\frac{|\nabla H|^2}{H^3}-\frac{|A|^2}H\nonumber\\
&\quad -\frac 1H\biggl(-n+1+m(n-2)\lambda_{\kappa}^{-n}-mn(n-2)\lambda_{\kappa}^{-n}|\pt_r^{\perp}|^2\biggr)\label{evl-H-Hn}\\
  \pt_t\chi=&\frac 1{H^2}\Delta \chi+\frac{|A|^2}{H^2}\chi+\frac 1{H^2}mn(n-2)\lambda_{\kappa}^{-n}|\pt_r^{\top}|^2\chi.\label{evl-chi-Hn}
\end{align}
Under the condition \eqref{lem-Hn-cond1}, from lemma \ref{lem-star-shape} and remark \ref{rem-3-2} we have
\begin{equation}\label{star-shape-Hn}
  R_1e^{\frac t{n-1}}\leq \chi\leq \lambda_{\kappa}(r)\leq R_2e^{\frac t{n-1}}.
\end{equation}
Define $u=(H\chi)^{-1}$. Since we have estimated the bounds of $\chi$ in \eqref{star-shape-Hn}, to estimate the lower bound of $H$ it suffices to estimate the upper bound of $u$. Combining the equations \eqref{evl-H-Hn} and \eqref{evl-chi-Hn}, we can derive the evolution equation of $u$:
\begin{align}\label{evl-u1}
  \pt_tu =& -H^{-2}\chi^{-1}\pt_tH-H^{-1}\chi^{-2}\pt_t\chi \nonumber\\
  =&-H^{-2}\chi^{-1}\left( \frac{\Delta H}{H^2}-2\frac{|\nabla H|^2}{H^3}-\frac{|A|^2}H\right)\nonumber\\
  &+H^{-3}\chi^{-1}\left(m(n-2)\lambda_{\kappa}^{-n}-(n-1)-mn(n-2)\lambda_{\kappa}^{-n}|\pt_r^{\perp}|^2\right)\nonumber\\
  & -H^{-1}\chi^{-2}\left(\frac 1{H^2}\Delta \chi+\frac{|A|^2}{H^2}\chi\right)-H^{-3}\chi^{-1}mn(n-2)\lambda_{\kappa}^{-n}|\pt_r^{\top}|^2\nonumber\\
  = &H^{-2}(\chi^{-1}\Delta H^{-1}+H^{-1}\Delta \chi^{-1})-2H^{-3}\chi^{-3}|\nabla \chi|^2\nonumber\\
  &-(n-1)H^{-3}\chi^{-1}-m(n-1)(n-2)\lambda_{\kappa}^{-n}H^{-3}\chi^{-1}\nonumber\\
  \leq &~ \textrm{div}(H^{-2}\nabla u)-2H^{-2}u^{-1}|\nabla u|^2-(n-1)u^3\chi^2,
\end{align}
where we used $m\geq 0$ and $ |\pt_r^{\perp}|^2+|\pt_r^{\top}|^2=|\pt_r|^2=1$.  To estimate the upper bound of $u$, let
\begin{equation}
  v=(t-t_0)^{\frac 12}u,
\end{equation}
where $t_0>0$ is arbitrary but fixed. Then $v(x, t_0)\equiv 0$ on $\Sigma_{t_0}$. From the evolution equation \eqref{evl-u1} for $u$, we derive that
\begin{align}\label{evl-v1}
  \pt_tv\leq &~ \textrm{div}(H^{-2}\nabla v)-2H^{-2}v^{-1}|\nabla v|^2+\frac 12(t-t_0)^{-1}v\nonumber\\
  &\quad -(n-1)(t-t_0)^{-1}v^3\chi^2.
\end{align}
Let $v_k=\max(v-k,0)$ for $k\geq 0$ and let $A(k)=\{x\in\Sigma_t, v(x,t)>k\}$. Then by \eqref{evl-measure} and \eqref{evl-u1},
\begin{align*}
  \frac d{dt}\int_{\Sigma_t}v_k^2d\mu_t\leq & \int_{A(k)}2v_k\nabla^i(H^{-2}\nabla_i v)d\mu_t-4\int_{A(k)}H^{-2}v_kv^{-1}|\nabla v|^2d\mu_t\\
  &+(t-t_0)^{-1}\int_{A(k)}vv_kd\mu_t +\int_{A(k)}v_k^2d\mu_t\\
  &-2(n-1)(t-t_0)^{-1}\int_{A(k)}v^3v_k\chi^2d\mu_t\\
\leq & (t-t_0)^{-1}\int_{A(k)}vv_kd\mu_t+\int_{A(k)}v_k^2d\mu_t\\
   &-2(n-1)(t-t_0)^{-1}\int_{A(k)}v^3v_k\chi^2d\mu_t.
\end{align*}
Using $v>k$ on $A(k)$ and $\chi\geq R_1e^{\frac t{n-1}}$, we have
\begin{align*}
  \frac d{dt}\int_{\Sigma_t}v_k^2d\mu_t\leq
   & (t-t_0)^{-1}\int_{A(k)}vv_kd\mu_t+\int_{A(k)}v_k^2d\mu_t\\
   &-2(n-1)k^2(t-t_0)^{-1}R_1^2e^{\frac{2t}{n-1}}\int_{A(k)}vv_kd\mu_t
\end{align*}
If for $0< t_0<t_1$,
\begin{equation}\label{cond-k}
  k^2\geq k_0^2=\frac 1{n-1}R_1^{-2}e^{-\frac {2t_0}{n-1}}\max\{t_1-t_0,1\},
\end{equation}
then on the interval $[t_0,t_1]$, we have
\begin{align}
  \frac d{dt}\int_{\Sigma_t}v_k^2d\mu_t\leq & 0.
\end{align}
Note that by the definition $v_k(x, t_0)\equiv 0$, then the above inequality implies that $v_k(x, t)\equiv 0$ for all $t\in [t_0, t_1]$, provided that $k$ satisfies \eqref{cond-k}.

We now consider the small and large times separately: If $t_1\leq 2$, we choose $t_0=t_1/2\leq 1$ and then
 \begin{equation*}
   k_0^2=\frac 1{n-1}R_1^{-2}e^{-\frac {2t_0}{n-1}}.
 \end{equation*}
From the definition of $v_k$, we conclude that
\begin{equation*}
  \sup_{\Sigma_{t_1}}u\leq t_0^{-\frac 12}(n-1)^{-\frac 12}R_1^{-1}e^{-\frac{t_0}{n-1}}\leq e^{\frac 1{n-1}}\sqrt{2}t_1^{-\frac 12}(n-1)^{-\frac 12}R_1^{-1}e^{-\frac{t_1}{n-1}},
\end{equation*}
since $t_0\leq 1$. If $t_1\geq 2$, we choose $t_0=t_1-1\geq 1$ and we still have
 \begin{equation*}
   k_0^2=\frac 1{n-1}R_1^{-2}e^{-\frac {2t_0}{n-1}}.
 \end{equation*}
In this case,
 \begin{equation*}
  \sup_{\Sigma_{t_1}}u\leq k_0=(n-1)^{-\frac 12}R_1^{-1}e^{-\frac{t_0}{n-1}}\leq e^{\frac 1{n-1}}(n-1)^{-\frac 12}R_1^{-1}e^{-\frac{t_1}{n-1}}.
\end{equation*}

In summary, in any case we have
 \begin{equation}
  \sup_{\Sigma_{t}}u\leq e^{\frac 1{n-1}}(n-1)^{-\frac 12}R_1^{-1}e^{-\frac{t}{n-1}}\max\{\sqrt{2}t^{-\frac 12},1\}.
\end{equation}
Then from the definition $u=(H\chi)^{-1}$ and \eqref{star-shape-Hn}, we obtain
\begin{align*}
  \min_{\Sigma_t}H\geq  & e^{-\frac 1{n-1}}(n-1)^{\frac 12}R_1R_2^{-1}\min\{\frac 1{\sqrt{2}}t^{\frac 12},1\}.
\end{align*}

\endproof
\begin{rem}
Note that the lower bound \eqref{H-sharp-bd} of $H$ is independent of the initial mean curvature, while the lower bound of $H$ in \cite{BHW,GWWX} depends on the mean curvature $H$ of the initial hypersurface.
\end{rem}

\begin{rem}
Here we only proved the estimate \eqref{H-sharp-bd} for inverse mean curvature flow in Kottler space with nonnegative mass. In the case that the Kottler space has negative mass $m<0$, one can easily check that $s_{\kappa}\geq (-m(n-2))^{1/n}$. Since we always have $\lambda_{\kappa}(r)\geq s_{\kappa}$, the two terms in the sixth line of \eqref{evl-u1} can be estimated as
\begin{align*}
  &-(n-1)H^{-3}\chi^{-1}-m(n-1)(n-2)\lambda_{\kappa}^{-n}H^{-3}\chi^{-1}\\
   \leq& -(n-1)H^{-3}\chi^{-1}-m(n-1)(n-2)s_{\kappa}^{-n}H^{-3}\chi^{-1} ~ \leq ~ 0.
\end{align*}
One can hope to explore the negative gradient term in \eqref{evl-u1} to estimate the upper bound of $u$ as Huisken-Ilmanen did in \cite{HI08}. If the initial hypersurface satisfies a slightly stronger condition $\lambda_{\kappa}(r)\geq 2^{1/n}s_{\kappa}$, then
\begin{align*}
  &-(n-1)H^{-3}\chi^{-1}-m(n-1)(n-2)\lambda_{\kappa}^{-n}H^{-3}\chi^{-1}\\
   \leq& -(n-1)H^{-3}\chi^{-1}-\frac 12 m(n-1)(n-2)s_{\kappa}^{-n}H^{-3}\chi^{-1} \\
  \leq & -\frac 12(n-1)H^{-3}\chi^{-1}.
\end{align*}
So we have enough negative term in the evolution equation of \eqref{evl-u1} and the remaining part in the proof of lemma \ref{lem-Hn-1} can be followed as well. Thus in this case, the estimate \eqref{H-sharp-bd} also holds with a slightly adjusted constant on the right hand side.
\end{rem}

Suppose $X: \Sigma\times [0,T)\ra (M_{\kappa}, \bar{g}_{\kappa})$ is a smooth solution to the inverse mean curvature flow which is star-shaped and the support function $\chi$ of the initial data $\Sigma_0$ satisfies \eqref{lem-Hn-cond1}. Then \eqref{H-sharp-bd} gives a lower bound of $H$ on all $\Sigma_t, t\in [0,T)$, independent on the initial mean curvature. In view of the evolution equation of $H$ in \eqref{evl-H-Hn}, the mean curvature is uniformly bounded above by its initial bound $\sup_{\Sigma_0}H$. Then as in \cite{BHW,GWWX}, we have that the full second fundamental form $A$ is uniformly bounded on $\Sigma_t, t\in [T/2,T)$. If $T<\infty$, the regularity results of Krylov \cite{Kry} gives the higher regularity of the solution and convergence to smooth limit hypersurface $\Sigma_T$ as $t\ra T$.  Then the short time existence result can be used to extend the solution smoothly past the time $T$. So we have $T=\infty$ and the solution exists for all time $[0, \infty)$.

To prove Theorem \ref{thm-H-geq0}, we need the following approximation lemma
\begin{lem}\label{lem-approx}
Let $X_0: \Sigma \ra (M_{\kappa}, \bar{g}_{\kappa})$ be a closed oriented hypersurface immersion of class $C^1$, with measurable, bounded and nonnegative weak mean curvature. Then $\Sigma_0=X_0(\Sigma)$ is of class $C^{1,\beta}\cap W^{2,p}$ for all $0<\beta<1, 1 \leq p < \infty$, and can be approximated locally uniformly in $C^{1,\beta}\cap W^{2,p}$ by a family of smooth hypersurfaces $\tilde{\Sigma}_{\epsilon}$, $0<\epsilon<\epsilon_0$, satisfying $H > 0$.
\end{lem}
\begin{proof}
The proof is similarly as that in \cite[Lemma 2.6]{HI08}, except that we are using the mean curvature flow in $(M_{\kappa}, \bar{g}_{\kappa})$ and the evolution equations of the mean curvature $H$ and the second fundamental form $|A|^2$ are (see \cite{Hu-86})
\begin{align*}
  \pt_tH= &\Delta H+(|A|^2+\overline{Ric}(\nu,\nu))H \\
  \pt_t|A|^2= &\Delta |A|^2-2|\nabla A|^2+2(|A|^2+\overline{Ric}(\nu,\nu))|A|^2\\
  &+4h^{ij}h_j^{m}\bar{R}_{mli}^{\quad \,\,l}-4h^{ij}h^{lm}\bar{R}_{milj}+2h^{ij}(\bar{\nabla}_j\bar{R}_{0li}^{\quad l}+\bar{\nabla}_l\bar{R}_{0ij}^{\quad l})
\end{align*}
The additional curvature terms in the above evolution equations are of lower order, so the argument in \cite[Lemma 2.6]{HI08} can be carried over directly and we omit the details.
\end{proof}

We now finish the proof of Theorem \ref{thm-H-geq0}. Given $\Sigma_0$, the initial hypersurface in Theorem \ref{thm-H-geq0},  let $\tilde{\Sigma}_{\epsilon}$ be the family of approximated hypersurfaces of positive mean curvature in Lemma \ref{lem-approx}. Starting from each $\tilde{\Sigma}_{\epsilon}, 0<\epsilon<\epsilon_0$, \eqref{icf} has a global smooth solution for all time $[0, \infty)$, all the estimates in Lemma \ref{lem-Hn-1} and the higher regularity estimates are uniform in $\epsilon$ for each positive fixed $t>0$. By letting $\epsilon\ra 0$ we obtain the desired global solution of \eqref{icf} for all $0<t<\infty$. This solution approaches the initial data $\Sigma_0$ uniformly as $t\ra 0$ in view of the estimate on $H$ in Lemma \ref{lem-Hn-1} and the fact that $\Sigma_t$ can be written as a graph of bounded gradient over $\Sigma_0$.


\bibliographystyle{Plain}

\end{document}